\documentclass[11pt,reqno,a4paper]{amsart}
\usepackage[english]{babel}
\usepackage[T1]{fontenc}
\usepackage[utf8]{inputenc}

\usepackage{lineno}
\usepackage{amsfonts}
\usepackage{mathtools}
\usepackage{verbatim}
\usepackage{mathtools}
\usepackage{dsfont}
\usepackage{amsthm}
\usepackage{indentfirst}
\usepackage{amsmath,amssymb}
\usepackage{verbatim}
\usepackage{array}
\usepackage{fancyhdr}
\usepackage{xcolor}
\usepackage{noindentafter}
\usepackage{esint}
\usepackage{mathrsfs}
\usepackage{stmaryrd}
\usepackage{environ}
\usepackage[hidelinks]{hyperref}
\usepackage{amsmath}
\usepackage{enumerate}
\usepackage{xfrac}
\usepackage{thmtools}

\NoIndentAfterCmd{\]}
\NoIndentAfterEnv{equation}
\NoIndentAfterEnv{align*}
\NoIndentAfterEnv{itemize}
\NoIndentAfterEnv{enumerate}

\newcommand{\weak}{\overset{*}{\rightharpoonup}}

\newcommand{\ped}{\beta,\delta,\varepsilon,x_0}

\DeclareMathOperator{\id}{id}

\newcommand{\R}{\mathbb{R}}
\newcommand{\Z}{\mathbb{Z}}

\newcommand{\N}{\mathbb{N}}

\DeclareMathOperator{\inttt}{int}

\DeclareMathOperator{\cof}{cof}
\DeclareMathOperator{\rank}{rank}

\DeclareMathOperator{\dx}{dx}

\DeclareMathOperator{\dv}{div}

\DeclareMathOperator{\spt}{spt}

\DeclareMathOperator*{\clos}{clos}
\DeclareMathOperator*{\intt}{int}

\DeclareMathOperator{\loc}{loc}
\DeclareMathOperator{\Sym}{Sym}

\DeclareMathOperator{\diag}{diag}

\DeclareMathOperator{\diam}{diam}

\theoremstyle{plain}
\newtheorem*{NTeo}{Theorem}

\newtheorem{Teo}{Theorem}

\newtheorem{lemma}[Teo]{Lemma}
\newtheorem{prop}[Teo]{Proposition}

\theoremstyle{definition}
\newtheorem{Def}[Teo]{Definition}
\theoremstyle{remark}
\newtheorem*{rem}{Remark}

\NoIndentAfterEnv{Def}
\NoIndentAfterEnv{lemma}
\NoIndentAfterEnv{prop}
\NoIndentAfterEnv{Cor}
\NoIndentAfterEnv{prop}
\NoIndentAfterEnv{rem}
\NoIndentAfterEnv{Rem}
\NoIndentAfterEnv{Teo}
\NoIndentAfterEnv{Theorem}
\NoIndentAfterEnv{DT}

\title{On a Question of D. Serre}

\author{Luigi De Rosa}
\address{L.D.R.: EPFL SB, Station 8, CH-1015 Lausanne, Switzerland}
\email{luigi.derosa@epfl.ch}

\author{Riccardo Tione}
\address{R.T.: Institut f\"ur Mathematik, Universit\"at Z\"urich, Winterthurerstrasse 190, CH-8057 Zurich, Switzerland}
\email{riccardo.tione@math.uzh.ch}

\begin{document}

\maketitle

\begin{abstract}
In this paper we give a negative answer to the question posed in \cite[Open Question 2.1]{SER} about possible gains of integrability of determinants of divergence-free, non-negative definite matrix-fields. We also analyze the case in which the matrix-field is given by the Hessian of a convex function.
\end{abstract}
\par
\medskip\noindent
\textbf{Keywords:} Matrix-fields, determinants, integrability.
\par
\medskip\noindent
{\sc MSC (2010): 26B25, 39B42, 39B62, 49N60.
\par
}
\section{Introduction}

The aim of this note is to answer \cite[Open Question 2.1]{SER}, that we recall here. Let $\Gamma$ be a lattice of $\R^n$ (one can imagine $\Gamma = \Z^n$, i.e. $\R^n / \Gamma = \mathbb{T}^{n }$, the $n$-dimensional torus), and we will denote by $\Sym^+(n)$ the space of $n\times n$ symmetric non-negative definite matrices.
\newline
\newline
\textbf{Open Question 2.1:} Let $x \mapsto A(x)$ be $\Gamma$-periodic, taking values in $\Sym^+(n)$. Let $A$ and $\dv(A)$ belong to $L^p(\R^n/\Gamma)$ with $1 < p < n$. Defining $\frac{1}{p'}=\frac{1}{p}-\frac{1}{n}$, is it true that $$\det(A)^{\frac{1}{n}} \in L^{p'}(\R^n / \Gamma) \,?$$
\newline
\noindent The answer to the question is negative, and our proof involves the construction of a family of "approximate counterexamples" in Lemma \ref{constr}, and then an application of Baire's category Theorem in Theorem \ref{Baire} to find actual counterexamples that are also topologically typical (in the sense of Baire). Even though in our case the situation is quite simple since our family of starting approximate counterexample is explicit, notice that these two steps are common to all the so-called convex integration schemes (compare, for instance, \cite[Proposition 4.17]{KIRK}). 
\\
\\
D. Serre's question was motivated by his results of \cite{SER}, that consisted in showing higher integrability of the determinant of $\Sym^+(n)$-valued fields $A$ when $\dv(A) = 0$ (or $\dv(A)$ is a bounded measure), and $A \in L^1$. We recall here \cite[Theorem 2.1]{SER}, but see also \cite[Theorem 2.2, 2.3, 2.4]{SER} for analogous results:

\begin{NTeo} Let the divergence-free, non-negative definite matrix field $x \mapsto A(x)$ be $\Gamma$-periodic, with $A \in L^1(\R^n/ \Gamma)$. Then
\[
\det(A) \in L^{\frac{1}{n -1}}(\R^n/\Gamma)
\]
and there holds
\[
\fint_{\R^n/\Gamma}\det(A(x))^{\frac{1}{n - 1}}\dx \le \det\left(\fint_{\R^n/\Gamma} A(x)\dx\right)^{\frac{1}{n - 1}}
\]
\end{NTeo}

The idea behind Serre's result, that can be read in \cite[Proposition 1.2]{SER} and the discussion before and after the proposition, is that if one couples a PDE constraint (such as $\dv(A) = 0$), with some other constraint (for instance, $A \in \Sym^+(n)$), it is possible to show some "elliptic property" of the solutions, i.e. some improvement in properties of the solutions. In the particular case when the operator is the divergence, one can show that, on matrices with maximal rank, $A \mapsto \det(A)$ is "elliptic" (for instance, see \cite[Proposition 1.2]{SER}). In the recent papers \cite{DIM,GUIANN}, it is shown that many problems, such as Alberti's Rank One Theorem \cite{ALB}, or the recent extension of Allard's rectifiability result \cite{DDG}, can be solved by similar considerations. These ideas have their root in the classical theory of compensated compactness of F. Murat and L. Tartar \cite{MUR1,MUR2,TAR1,TAR2}.
\\
\\
Let us outline the structure of this paper and the strategy of the proof. In the first section, we give the answer to Serre's question. In Lemma \ref{constr}, we construct a family of convex functions, that we then use in Theorem \ref{Baire} to construct the required counterexample. The key observation here is that the cofactor matrix of a (sufficiently regular) gradient is always divergence-free, as proved, for instance, in \cite[Ch. 8, Th. 2]{EVAPDE}. In the second and final section, we use again the family of Lemma \ref{constr} to discuss the case in which the matrix is the Hessian of a convex function. In particular, we show also in this case that in general one does not have any gain of integrability of the determinant in both Lebesgue $L^p$ and Hardy $\mathcal{H}^1$ spaces.

\subsection{Notation and preliminary results}

We will always use the euclidean norm on matrices $A \in \R^{n\times m}$ and vectors $v \in \R^n$, denoted with $\|A\|$ and $\|v\|$. For any matrix $A \in \R^{n\times n}$, we denote with $\cof(A)$ the matrix obtained as
\[
\cof(A)_{ij} = (-1)^{i + j}\det(A^{ji}),
\]
where $A^{ji}$ is the $(n-1)\times(n -1)$ matrix obtained from $A$ by eliminating the $j$-th row and the $i$-th column. In particular
\[
A\cof(A) = \det(A)\id_n, \quad\forall A \in \R^{n\times n}.
\]
For functions $f: \R^n \to \R$, we denote with $\nabla f$ its gradient, and with $Hf$ its Hessian, i.e. the matrix of the second derivatives of $f$. With $\det(Hf)$ we always denote the "pointwise" determinant, i.e. the determinant of the classical Hessian matrix (we recall that the Hessian matrix of a convex function is defined a.e., as proved in \cite[Theorem 6.9]{EVG}). Finally, we will use the concept of Monge-Amp\`ere measure associated to a convex function. For every convex function $\varphi: \Omega \to \R$, this is defined as the locally finite measure:
\[
\mu_\varphi(E) \doteq \left|\bigcup_{x \in E}\partial \varphi(x)\right|,
\]
where $\partial \varphi(x)$ denotes the subdifferential at $x$ of $\varphi$ and $|A|$ is the Lebesgue measure of the set $A$. We refer the reader to \cite[Section 2]{FIG} for the basic properties of $\mu_\varphi$. The terms "absolutely continuous" and "singular" part of a measure need to be intended with respect to the Lebesgue measure. We denote the singular part of a measure $\mu$ as $\mu^s$. Moreover, for a Borel set $E\subset \R^n$ we use the symbol $\mu\llcorner E$ to denote the measure
\[
\mu\llcorner E(A) \doteq \mu(E\cap A), \quad \forall A \text{ Borel subset of }\R^n.
\]

We will denote by $\mathcal{H}_{\loc}^1(\Omega)$ the local Hardy space. We will just need to consider non-negative functions in this space, and we recall that for a measurable, non-negative function $f:\Omega \rightarrow \R$, $f \in \mathcal{H}_{\loc}^1(\Omega)$ if and only if (see \cite[Lemma 3]{MULDETPOS}, that is a consequence of \cite{STE})
$$
\|f\|_{\mathcal{H}^1(K)}=\int_{K} f(x) \log(1+f(x))\, \dx\, < +\infty,\; \forall K \subset \Omega, K \text{ compact}.
$$
\subsection*{Aknowledgements} We are grateful to Guido De Philippis for his interest and comments on the problem and for reading a preliminary draft of this note.
\section{The counterexample}

\noindent Let  $\Omega$ be an open subset of $\R^n$. Let 
\begin{align*}
Y_{p,K} \doteq \{&A \in L^p(\Omega,\Sym^+(n)): \dv(A) \in L^p(\Omega,\R^n), \\ &A \equiv \bar{A} \text{ outside } K, \text{ for some fixed $\bar A \in \Sym^+(n)$}\}, 
\end{align*}
for any compact $K \subset \Omega$ with $\clos(\inttt(K)) = K \neq \emptyset$. We consider the following distance on $Y_{p,K}$, that turns it into a complete metric space:
\[
d(A,B) \doteq \|A - B\|_{L^p} + \|\dv(A - B)\|_{L^p}.
\]
We prove the following
\begin{Teo}\label{Baire}
Let $p^*\doteq \max\left\{0, \frac{p(n-1) - n}{p(n-1)}\right\}$. The set $$D_{p,K} \doteq \{A \in Y_{p,K}: \det(A)^{\frac{1}{n-1}} \in L^{\frac{1}{1 - p^*}}(\Omega)\setminus L^{\frac{1}{1 - p^*} + \varepsilon}(\Omega), \forall \varepsilon > 0\}$$ is residual in $Y_{p,K}$.
\end{Teo}

\begin{Def}
$p^* \doteq \max\left\{0, \frac{p(n-1) - n}{p(n-1)}\right\}$ is the \emph{critical exponent}.
\end{Def}

\begin{rem} The same result (without modifying the proof) would have held if we had required $\dv(A) = 0$ in the definition of $Y_{p,K}$, or if we had chosen instead of $Y_{p,K}$,
\begin{align*}
X_p = \{&A \in L^p(\Omega,\Sym^+(n)): \dv(A) \in L^p, \\
&A\nu \equiv \bar{A}\nu \text{ on }\partial\Omega, \text{ for some fixed $\bar A \in \Sym^+(n)$}\},
\end{align*}
or, as in Serre's original question
\[
S_p = \{A \in L^p(\R^n / \Gamma, \Sym^+(\R^n)): \dv(A) \in L^p(\R^n / \Gamma, \Sym^+(\R^n))\},
\]
\end{rem}
\begin{rem}
Let us explain how this result gives negative answer to \cite[Open Question 2.1]{SER}. If $p \le \frac{n}{n -1}$, then $p^* = 1$, and we obtain the existence of one (in fact, many) divergence free, non-negative definite tensor fields $A$ such that
\[
\det(A)^{\frac{1}{n - 1}} \in L^1\setminus L^{1 +\varepsilon}, \quad \forall \varepsilon >0, 
\]
thus proving the optimality of Serre's results. Moreover, also in the supercritical case, i.e. $p > \frac{n}{n - 1}$ that yields
\[
\frac{1}{1-p^*} = p\frac{n}{n - 1},
\]
 
\noindent Theorem \ref{Baire} tells us that for many divergence free, non-negative definite $A$, $\det^{\frac{1}{n - 1}}(A) \in L^{\frac{p(n - 1)}{n}} \setminus L^{\frac{p(n - 1)}{n}+\varepsilon} $, thus proving that there can be no general gain in the integrability of the determinant with respect to the general estimate $\det(A) \in L^{\frac{p}{n}}$. 
\end{rem}

\begin{lemma}\label{constr}
Fix $p \ge 1$. For every $\beta >0,\delta> 0,\varepsilon > 0,  x_0 \in \Omega$ there exists a convex function $\varphi_{\beta,\delta,\varepsilon,x_0} \in W_{\loc}^{2,p(n-1)}(\Omega)$ and a matrix $S_{\beta,\delta,\varepsilon,x_0} \in \Sym^+(n)$ such that
\begin{enumerate}[(i)]
\item $\varphi_{\beta,\delta,\varepsilon,x_0} \equiv x^TS_{\beta,\delta,\varepsilon,x_0}x$ outside $B_\beta(x_0)$\label{spt};
\item $\|\cof(H\varphi_{\ped})\|_{L^p(\Omega)} \le  \delta$\label{piccolezza};
\item $\det^{\frac{1}{n - 1}}(\cof(H\varphi_{\ped})) \notin L^{\frac{1}{1 - p^*} + \varepsilon}(B_r(x_0)), \forall r > 0$.
\label{esplosione}
\end{enumerate}
\end{lemma}
\begin{proof} We divide the proof in four steps:
\newline
\newline
\fbox{Step 1: Definition and properties of the starting function.}
\newline
\newline
\indent For $\alpha \ge 0$, define the function
\[
f_\alpha(x)\doteq
\begin{cases}
\|x\|^{1 + \alpha} + b, &\text{ if } \|x\| \le 1,\\
a\|x\|^2, &\text{ if } \|x\| > 1,
\end{cases}
\]
where $a,b \in \R$ are chosen in such a way that the function $f_\alpha$ is in $C^1(\R^n\setminus\{0\})$, i.e. we need to solve
\[
1 + b = a \text{ and } 1 + \alpha= 2a.
\]
Therefore
\begin{equation}\label{func}
f_\alpha(x)\doteq
\begin{cases}
\|x\|^{1 + \alpha} +\frac{ \alpha - 1}{2}, &\text{ if } \|x\| \le 1,\\
\frac{1+\alpha}{2}\|x\|^2, &\text{ if } \|x\| > 1.
\end{cases}
\end{equation}
It is easy to see that $f_\alpha$ defined in this way is convex. We compute its pointwise Hessian (except for the points $x \in \R^n$ such that $\|x\| = 0$ or $\|x\| = 1$):

\begin{equation}\label{hess}
Hf_\alpha(x)\doteq
\begin{cases}
(1 + \alpha)\left(\|x\|^{\alpha - 1}\id_n + (\alpha - 1)\|x\|^{\alpha -3}x\otimes x\right), &\text{ if } 0 < \|x\| < 1,\\
(1+\alpha) \id_n, &\text{ if } \|x\| > 1.
\end{cases}
\end{equation}
\newline
\newline
\indent \fbox{Step 2: $L^p$ estimates on $Hf_\alpha$.}
\newline
\newline
\indent We can estimate, for some constant $C_{\alpha,n} > 0$,
\begin{equation}\label{est1}
\|Hf_\alpha\|(x) \le
\begin{cases}
C_{\alpha,n}\|x\|^{\alpha - 1}, &\text{ if } 0 <\|x\| < 1,\\
(1+\alpha)\sqrt{n}, &\text{ if } \|x\| > 1.
\end{cases}
\end{equation}
Moreover, recalling the following linear algebra result, often called Matrix Determinant Lemma
\[
\det(A + B) = \det(A) + \langle B,\cof^T(A)\rangle,\quad \forall A,B \in \R^{n\times n},\; \rank(B) = 1,
\]
we compute:
\begin{equation}\label{est2}
\det(Hf_\alpha)(x) =
\begin{cases}
\alpha(1 + \alpha)^n\|x\|^{n(\alpha -1)}, &\text{ if } 0 <\|x\| < 1,\\
(1+\alpha)^n, &\text{ if } \|x\| > 1.
\end{cases}
\end{equation}
From \eqref{est1}, we find that $Hf_\alpha \in L_{\loc}^p(\R^n)$ for every $\alpha \ge 0$ if $p < n$ and for $\alpha > \frac{p - n}{p}$ if $p \ge n$. For these values of $\alpha$, we also get $f_\alpha \in W^{2,p}_{\loc}(\R^n)$, as proved in Lemma \ref{Sobolev}, and that $Hf_\alpha$ is not only the pointwise Hessian of $f_\alpha$ but also its distributional Hessian.
\newline
\newline
\fbox{Step 3: Integrability of the determinant and the cofactors of $Hf_\alpha$.}
\newline
\newline
\indent Define
\begin{equation}\label{A}
A_\alpha(x) \doteq \cof(H f_\alpha)(x).
\end{equation}
In view of the equality $\det^{\frac{1}{n - 1}}(A_\alpha) = \det(Hf_\alpha)$ and \eqref{est2},
\begin{equation}\label{expl}
\det(A_\alpha)^{\frac{1}{n - 1}} \in L_{\loc}^{\frac{1}{1 - \alpha} - \varepsilon}(\R^n), \quad \forall \varepsilon > 0,
\end{equation}
but 
\begin{equation}\label{expl}
\det(A_\alpha)^{\frac{1}{n - 1}} \notin L^{\frac{1}{1 - \alpha}}(B_r(0)) \text{ for any } r > 0. 
\end{equation}
Moreover, by \eqref{est1},
\begin{align*}
\|A_\alpha\|(x) &= \|\cof(H f_\alpha)\|(x)\le c_{n}\|H f_\alpha\|^{n - 1}(x)  \overset{\eqref{est1}}{\le} C'_{\alpha,n}\max\{\|x\|^{(n -1)(\alpha - 1)},1\},
\end{align*}
for some constant $C'_{\alpha,n} > 0$. Hence, if $(n -1)(1 - \alpha)p < n$, i.e. if $\alpha > p^*$, then $A_\alpha \in L_{\loc}^p(\R^n)$. The same computation shows, in particular, that for $\alpha > p^*$ one has $f_\alpha \in W_{\loc}^{2,p(n-1)}(\R^n)$.
\newline
\newline
\fbox{Step 4: Construction of $\varphi_{\ped}$.}
\newline
\newline
\indent Fix $p,\beta,\delta, \varepsilon, x_0$ as in the statement of the Lemma. Choose $\alpha = \alpha(\varepsilon) > 0$ such that
\[
 \frac{1}{1 - \alpha} = \frac{1}{1 - p^*}+ \varepsilon,
\]
that in particular implies $\alpha > p^*$. Finally define, for a constant $c_{\beta,\delta,\varepsilon} > 0$ to be fixed later,
\begin{equation}\label{param}
\varphi_{\ped}(x) \doteq c_{\beta,\delta,\varepsilon}\left[f_{\alpha}\left(\frac{2}{\beta} (x - x_0)\right) -2\left(\frac{1 + \alpha}{\beta^2}\right)(\|x_0\|^2 -2(x,x_0))\right].
\end{equation}
By the definition of $f_\alpha$, we get $\eqref{spt}$. Moreover, \eqref{esplosione} is a consequence of our choice of $\alpha$ and \eqref{expl}. Finally, since $\alpha > p^*$, $A_\alpha$ belongs to $L_{\loc}^p(\R^n)$, as proved in the previous step. Therefore, we can choose $c_{\beta,\delta,\varepsilon}$ small enough so that \eqref{piccolezza} is fulfilled.
\end{proof}

\begin{lemma}\label{Sobolev}
The function $f_\alpha$ defined in \eqref{func} is in $W_{\loc}^{2,p}(\R^n)$  for every $\alpha \ge 0$ if $p < n$ and for $\alpha > \frac{p - n}{p}$ if $p \ge n$. Moreover, its pointwise Hessian, computed in \eqref{hess}, coincides with its distributional Hessian.
\end{lemma}
\begin{proof}
To see this, we write, for any $\eta \in C_c^\infty(\R^n)$ and $i,j \in \{1,\dots,n\}$,
\begin{align*}
\int_{\R^n}f_\alpha \partial^2_{ij}\eta  = \lim_{R \to 0} \left[ \int_{\R^n\setminus (B_R(0)\cup S_R)}f_\alpha \partial^2_{ij}\eta\right]\,,
\end{align*}
where $S_R=B_{1-R}^c(0) \cap B_{1+R}(0)$. Integrating by parts we get
$$
\int_{\R^n}f_\alpha \partial^2_{ij}\eta  = \lim_{R \to 0}  \left[ \int_{\partial S_R \cup \partial B_R(0)}f_\alpha\, \partial_i \eta\, \nu^j -\int_{\R^n\setminus (B_R(0)\cup S_R)}\partial_j f_\alpha \partial_{i}\eta \right]\,,
$$
and since $f_\alpha \in C^0(\R^n)$ the first term vanishes. Thus we are left with the second one, which again integrating by parts can be written as
$$
\lim_{R \to 0} -\int_{\R^n\setminus (B_R(0)\cup S_R)}\partial_j f_\alpha \partial_{i}\eta = \lim_{R \to 0} \left[ -\int_{\partial S_R \cup \partial B_R(0)}\partial_j f_\alpha\,  \eta\, \nu^i + \int_{\R^n\setminus (B_R(0)\cup S_R)}\eta\, \partial^2_{ij} f_\alpha \right]\,.
$$
Note that for every $\alpha \geq 0$, $\partial_j f_\alpha \in L^\infty_{\loc}(\R^n)$ and $\partial_j f_\alpha  $ is continuous in $\R^n \setminus \{ 0\}$. Thus we have
$$
 \lim_{R \to 0}  \int_{\partial S_R \cup \partial B_R(0)}\partial_j f_\alpha\,  \eta\, \nu^i =\lim_{R \to 0} \left[  \int_{\partial S_R }\partial_j f_\alpha\,  \eta\, \nu^i + \int_{\partial B_R(0) }\partial_j f_\alpha\,  \eta\, \nu^i \right]=0\,.
$$
Finally it is clear that $\partial^2_{ij} f_\alpha$  is in $L_{\loc}^{p}(\R^n)$  for every $\alpha \ge 0$ if $p < n$ and for $\alpha > \frac{p - n}{p}$ if $p \ge n$
Thus, for the ranges of $\alpha$ and $p$ we are considering, we have $H f_\alpha \in L^p_{\loc}(\R^n)$, and by dominated convergence we conclude
$$
\int_{\R^n}f_\alpha \partial^2_{ij}\eta = \lim_{R \to 0} \int_{\R^n\setminus (B_R(0)\cup S_R)}\eta\, \partial^2_{ij} f_\alpha= \int_{\R^n} \eta\, \partial^2_{ij}f_\alpha\,.
$$
\end{proof}

We can finally prove our main result.

\begin{proof}[Proof of Theorem \ref{Baire}]
First observe that
\[
D_{p,K}^c = \{A \in Y_{p,K}: \det(A)^{\frac{1}{n-1}} \in L^{\frac{1}{1 - p^*} + \varepsilon} \text{ for some } \varepsilon > 0\},
\]
which is true for $p < \frac{n}{n -1}$ because of Serre's result \cite[Theorem 2.4]{SER}, while for $p \ge \frac{n}{n -1}$ it is just a consequence of the definition of $p^*$ and the fact that $\det(A)^{\frac{1}{n - 1}}\in L^{\frac{p(n - 1)}{n }}, \forall A \in Y_{p,K}$.\\

\noindent We want to write $D_{p,K}^c$ as a countable union of closed sets with empty interior. To do so, consider
\[
C_{k,j} = \{A \in Y_{p,K}: \|\det(A)^{\frac{1}{n - 1}}\|_{\frac{1}{1 - p^*} + \frac{1}{k}} \le j\}.
\]
For every $k,j$, $C_{k,j}$ is closed in $(Y_{p,K},d)$, as can be easily seen through Fatou's Lemma. Moreover, $$\bigcup_{k,j}C_{k,j} = D_{p,K}^c.$$
Finally, suppose that for some $k,j$, $C_{k,j}$ has non-empty interior. This means that we can find $\bar A \in C_{k,j}$ and a ball (in the $d$-topology on $Y_{p,K}$) of radius $\rho$, $\mathcal{N}_{\rho}(\bar A)$, such that $\mathcal{N}_\rho(\bar A) \subset C_{k,j}$. In particular this implies that 
\begin{equation}\label{contr}
\det(B)^{\frac{1}{n - 1}} \in L^{\frac{1}{1 - p^*} + \frac{1}{k}}(\Omega), \forall B \in \mathcal{N}_{\rho}(\bar A).
\end{equation}
Fix $x_0 \in \intt(K)\subset\Omega$ and let $r > 0$ be such that $B_r(x_0) \subset \intt(K)$. Consider $\varphi_{\beta,\delta,\varepsilon,x_0}$ of Lemma \ref{constr}, with $\varepsilon = \frac{1}{k}$, $\beta = \frac{r}{2}$ and $\delta = \frac{\rho}{2}$. Define also $$\overline{M}_{\beta,\delta,\varepsilon,x_0}(x) \doteq \cof(H\varphi_{\beta,\delta,\varepsilon,x_0}),$$ and finally take
\[
B \doteq \bar A + \overline{M}_{\beta,\delta,\varepsilon,x_0}.
\]
Observe that $\overline{M}_{\beta,\delta,\varepsilon,x_0}$ is a divergence-free non-negative definite tensor field, that is constant outside $K$. The matrix-field $\overline{M}_{\beta,\delta,\varepsilon,x_0}$ is divergence-free is because it is the cofactor matrix of the Hessian of a map $\varphi \in W_{\loc}^{2,p(n - 1)}(\R^n)$. Therefore, our choice of $\beta$ and $\delta$ imply that $B \in \mathcal{N}_\rho(\bar A)$. Hence \eqref{contr} implies $$\det(B)^{\frac{1}{n - 1}} \in L^{\frac{1}{1 - p^*} + \frac{1}{k}}.$$ Since the determinant is monotone on the cone of non-negative symmetric matrices, we have
\[
\det(B) = \det(\bar A +\overline{M}_{\beta,\delta,\varepsilon,x_0}) \ge \det(\overline{M}_{\beta,\delta,\varepsilon,x_0}) \ge 0,
\]
that would imply $\det(\overline{M}_{\beta,\delta,\varepsilon,x_0})^{\frac{1}{n - 1}} \in L^{\frac{1}{1 - p^*} + \frac{1}{k}}(B_\beta(x_0))$ but this contradicts \eqref{esplosione} of Lemma \ref{constr} by our choice of $\varepsilon$.
 \end{proof}

\begin{rem}
We conclude the section by noticing that the situation for diagonal matrices is less rigid. If $A = \diag(f_1,\dots,f_n)$, $f_i \in L^p(\R^n)$, compactly supported, and $\dv(A) \in L^p(\R^n)$, then $\det^{\frac{1}{n-1}}(A) \in L^{p}(\R^n)$, and
\[
\|( \det A)^{\frac{1}{n-1}}\|_{L^{p}} \le C\|\dv(A)\|^{\frac{n}{n-1}}_{L^p}\,,
\]
for some constant $C>0$ which depends on the size of the support of $A$.
Note that one does not even need the non-negativity of $A$ to be satisfied. The proof of the inequality is as follows. We have that $\partial_i f_i \in L^p(\R^n)$. Therefore
\[
|f_i|(x_1,\dots,x_n) = \left|\int_{-\infty}^{x_i}\partial_if_i(x_1,\dots,x_{i - 1},t,x_{i + 1},\dots,x_n)dt\right|
\]
and
\[
|f_i|^p(x_1,\dots,x_n) \le C\int_{-\infty}^{\infty}|\partial_i f_i|^p(x_1,\dots,x_{i - 1},t,x_{i + 1},\dots,x_n)dt,
\]
where $C = C(p,\diam(\spt(A)))$.
Define
\[
g_i(\hat x_i) \doteq \int_{-\infty}^{\infty}|\partial_i f_i|^p(x_1,\dots,x_{i - 1},t,x_{i + 1},\dots,x_n)dt.
\]
We have $g_i \in L^{1}(\R^{n - 1})$, hence $g_i^{\frac{1}{n - 1}} \in L^{n - 1}$. Therefore
\begin{align*}
\int_{\R^n}\det(A)^{\frac{p}{n - 1}}(x)\dx &\le C\int_{\R^n}\prod_i g^{\frac{1}{n - 1}}_i(\hat x_i) \dx \\
&\le C\prod_i\|g_i\|^{\frac{1}{n - 1}}_{L^1(\R^{n - 1})} \le C\|\dv(A)\|^{\frac{np}{(n - 1)}}_{L^p}.
\end{align*}
The second inequality can be found in \cite[Lemma 9.4]{BRE}. Since this inequality is sharp, it is easy to find counterexamples to the statement $\det(A) \in L_{\loc}^{q}(\R^n)$ for exponents $q > \frac{p}{n - 1}$.
\end{rem}

\section{Hessians}

In this section we consider the case of Hessians of convex functions. The analogy with the result of \cite{SER} is that, instead of \emph{divergence-free} tensor fields $A$, here we consider \emph{curl-free} tensor fields $A$. The  \emph{curl-free} assumption, together with the  symmetry of $A$, defines the class of Hessians of functions. Once we also add the non-negativity of the eigenvalues, we are lead to consider exactly Hessians of convex functions. We ask the following question: given $\Omega \subset \R^n$, an open and convex set, let $\varphi \in W_{\loc}^{2,p}(\Omega)$, $p \in [1,+\infty)$, be a convex function. What can be said about the integrability of $\det(H\varphi)$? \\

The case $p = n$ has been covered (in a more general setting) by S. M\"uller in \cite{MULDETPOS,MULDET} (see also \cite{DET} for the same result for mappings with determinants of arbitrary sign). More precisely, it is proved in \cite[Theorem 1]{MULDETPOS} that
$$\det(H\varphi) \in \mathcal{H}_{\loc}^1(\Omega)$$ and moreover that this is optimal in the following sense. In \cite[Counterexample 7.2]{MULDET}, M\"uller finds a sequence of maps $u_j \in W^{1,n}_{\loc}(\R^n)$ such that for every function $\gamma : \R^+ \to \R^+$ such that
\[
\frac{\gamma(z)}{z\log(1 + z)} \to +\infty, \text{ as } z \to \infty,
\]
one has
\[
\|\gamma(\det(\nabla u_j))\|_{L^1(B_1(0))} \to +\infty, \text{ as } j\to \infty.
\]
It is immediate to see that this sequence $u_j$ is actually $u_j = \nabla\varphi_j$, for some convex function $\varphi_j \in W_{\loc}^{2,n}(\R^n)$, hence M\"uller's results close the question in the case $p = n$. Theorem \ref{summ} answers the question in the case $p \in [1, \infty)\setminus\{n\}$. Let us first introduce the following space: for any compact set $K \subset \Omega$, with $\clos(\inttt(K)) = K \neq \emptyset$,
 \begin{align*}
H_{p,K} \doteq \{&\varphi \in W^{2,p}(\Omega): \varphi \text{ is convex, } H\varphi \equiv \bar{A} \text{ outside } K,\\
& \text{ for some fixed $\bar A \in \Sym^+(n)$}\}.
\end{align*}
This is a complete metric space when endowed with the distance
\[
d(\varphi_1,\varphi_2) \doteq \|\varphi_1 - \varphi_2\|_{W^{2,p}(\Omega)}.
\]
\begin{Teo}\label{summ}
The following hold
\begin{enumerate}[(i)]
\item If $p \in [1,n)$, then $\forall \varphi \in W^{2,p}(\Omega)$ convex, $\det(H\varphi) \in L^1_{\loc}(\Omega)$, but there exists a convex function $\bar\varphi \in W^{2,p}(\Omega)$ such that $\det(H\bar\varphi) \in L^{1}_{\loc}(\Omega)\setminus\mathcal{H}_{\loc}^1(\Omega)$;\label{<n}
\item If $p \in (n,+\infty)$, then $\forall \varphi \in W^{2,p}(\Omega)$ convex, $\det(H\varphi) \in L^{\frac{p}{n}}_{\loc}(\Omega)$, but there exists a convex function $\bar\varphi \in W^{2,p}(\Omega)$ such that $\det(H\bar\varphi) \in L^{\frac{p}{n}}_{\loc}(\Omega)\setminus L^{\frac{p}{n} + \varepsilon}_{\loc}(\Omega),\forall \varepsilon > 0$;\label{>n}
\end{enumerate}
\end{Teo}
\begin{proof}
First we deal with the "positive" part of the statement of the Theorem. We start by noticing that, for every convex function $\varphi$, $\det(H\varphi)$ is the density of the absolutely continuous part of the Monge-Amp\`ere measure $\mu_\varphi$ associated to $\varphi$ (see for instance \cite[Lemma 1.18]{GUIPHD}). By the Radon-Nikodym Theorem we have
$$
\det(H\varphi) \in L^1_{\loc}(\Omega) \qquad \forall \, \varphi \,\text{ \,convex}
$$
If we take a convex function  $\varphi \in W^{2,p}_{\loc}(\Omega)$ for some $p > n$, by H\"older inequality we trivially have $\det(H\varphi) \in L^{\frac{p}{n}}_{\loc}(\Omega) $. Let us now show the optimality of these results. The optimality for the case $\eqref{<n}$ is the content of Proposition \ref{last?}. To find a convex function $\bar\varphi \in W_{\loc}^{2,p}(\Omega)$, $p > n$, such that $\det(H\bar{\varphi}) \in L_{\loc}^{\frac{p}{n}}(\Omega)\setminus L_{\loc}^{\frac{p}{n} + \varepsilon}(\Omega)$ for every $\varepsilon > 0$, consider again the family of functions $f_\alpha$ defined in \eqref{func}. As proved in Step 2 of Lemma \ref{constr} and Lemma \ref{Sobolev}, we find that if $\alpha >\frac{p - n}{p}$, then $f_\alpha \in W_{\loc}^{2,p}(\R^n)$ and for every $\varepsilon > 0$, we find $\alpha = \alpha(\varepsilon) > 0$ such that $f_\alpha \in W_{\loc}^{2,p}(\R^n)$ but
\[
\det(Hf_\alpha) \notin L^{\frac{p}{n} + \varepsilon}(B_r(0)), \text{ for any } r>0.
\]
With a construction analogous to the one of Lemma \ref{constr} and the same proof as in Theorem \ref{Baire}, it is possible to prove that the set $$\{\varphi \in H_{p,K}: \det(H\varphi) \in L^{\frac{p}{n}}(\Omega)\setminus L^{\frac{p}{n} + \varepsilon}(\Omega), \forall \varepsilon > 0\}$$ is residual in $H_{p,K}$. By Baire's theorem, we then deduce the existence of such a function $\bar\varphi$.
\end{proof}

We will now prove the optimality of \eqref{<n} of Theorem \ref{summ}, namely

\begin{prop}\label{last?}
Let $p \in [1,n)$. The set $$U_{p,K}\doteq\{\varphi \in H_{p,K}: \det(H\varphi) \in L^{1}(\Omega)\setminus \mathcal{H}^1(\Omega)\}$$ is residual in $H_{p,K}$.
\end{prop}

To prove Proposition \ref{last?}, we first need the following result, for which we thank Guido De Philippis.

\begin{lemma}\label{MA}
Let $\varphi: \Omega \to \R$ be a convex function such that its Monge-Amp\`ere measure $\mu_\varphi$ has a non-trivial singular part with respect to the Lebesgue measure. Then, for every sequence of smooth and convex functions $\varphi_j$ converging locally uniformly to $\varphi$ and for every $z_0 \in \spt(\mu_\varphi^s)$, we have
\[
\|\det(H\varphi_j)\|_{\mathcal{H}^1(B_r(z_0))} \to + \infty, \text{ as $j\to \infty$},
\]
for every $B_r(z_0)$ compactly contained in $\Omega$.
\end{lemma}
\begin{proof}
By contradiction, suppose there exists a sequence $(\varphi_j)_j$, a point $z_0 \in \spt(\mu^s_{\varphi})$ and $r > 0$ as in the statement such that
\[
\sup_j\|\det(H\varphi_j)\|_{\mathcal{H}^1(B_r(z_0))} < + \infty.
\]
Equi-boundedness in $\mathcal{H}^1$ tells us that we can also assume, up to a non-relabeled subsequence, that $\det(H\varphi_j)$ converges weakly in $L^1$ to a function $F \in L^1(B_r(z_0))$ (see \cite[Theorem 4.1]{MULDET} for a proof). By definition of the Monge-Amp\`ere measure and the regularity of $\varphi_j$, we can write, $\forall f \in C^0_c(B_r(z_0))$,
\[
\int_{\Omega}f(x)d\mu_{\varphi_j}(x) = \int_{\Omega}f(x)\det(H\varphi_j)(x)\dx.
\]
Now, the uniform convergence $\varphi_j \to \varphi$ implies that $\mu_{\varphi_j} \weak \mu_\varphi$ (see \cite[Proposition 2.9]{FIG}), and the weak convergence of $\det(H\varphi_j)$ to $F$ combined with the previous equality implies that, in the limit,
\[
\int_{\Omega}f(x)d\mu_\varphi(x) = \int_\Omega F(x)f(x)\dx\,,\quad  \forall f \in C^0_c(B_r(z_0)).
\]
The last equality implies $\mu_\varphi\llcorner B_r(z_0) = F\chi_{B_r(z_0)}\dx $, contradicting the fact that $z_0 \in \spt(\mu^s_\varphi)$.
\end{proof}

\begin{proof}[Proof of Proposition \ref{last?}]
Fix $p \in [1,n)$. We consider the function $f_0$ constructed in the Steps 1 and 2 of Lemma \ref{constr}. By Lemma \ref{Sobolev}, $f_0\in W_{\loc}^{2,p}(\R^n)$  Analogously to Step 4 of the same lemma, for every $\beta,\delta,\varepsilon > 0$ and $x_0 \in \intt(K)$, we consider $\varphi_{\beta,\delta,\varepsilon,x_0}$ defined as in \eqref{param}. We choose $x_0$ arbitrarily and $\beta$ such that $B_{2\beta}(x_0) \subset K$. For this proof we will not need $\varepsilon$, that we consider fixed. Therefore, we will write $\varphi_\delta$ instead of $\varphi_{\beta,\delta,\varepsilon,x_0}$ for the sake of readability. To prove the Proposition, we write $U_{p,K}^c$ as the countable union of closed sets:
\[
U_{p,K}^c \doteq \bigcup_{m \in \N}\left\{\varphi \in H_{p,K}: \|\det(H\varphi)\|_{\mathcal{H}^1(\Omega)} \le m\right\}.
\]
Each set $C_m \doteq \left\{\varphi \in H_{p,K}: \|\det(H\varphi)\|_{\mathcal{H}^1(\Omega)} \le m\right\}$ is closed. To prove that it has empty interior we reason by contradiction. Therefore, we find $m$, $\rho > 0$ and $\bar\varphi$ such that the ball $\mathcal{N}_\rho(\bar\varphi) \subset C_m$. Now choose $\delta > 0$ in such a way that $\|\varphi_{\delta}\|_{W^{2,p}(\Omega)} \le \frac{\rho}{2}$. This can be done in view of \eqref{est1} (in the case $\alpha=0$). If we now mollify $\varphi_\delta$, we get a sequence of smooth convex functions $\varphi_{\delta,j} \in H_{p,K}$  such that $\|\varphi_{\delta,j}\|_{W^{2,p}(\Omega)}\le \frac{\rho}{2}, \forall j \in \mathbb{N}$. This sequence is also converging locally uniformly to $\varphi_{\delta}$, since it is well known that real-valued convex functions are locally Lipschitz. By the definition of $\varphi_\delta$ in \eqref{param} and the fact (see \cite[Example 2.2(2)]{FIG}) that
\[
\mu_{f_0}\llcorner B_1(0) = \omega_n\delta_0,
\]
we find that $x_0 \in \spt(\mu^s_{\varphi_{\delta}})$ and, by Lemma \ref{MA}, that
\begin{equation}\label{expl2}
\|\det(H\varphi_{\delta,j})\|_{\mathcal{H}^1(\Omega)} \to +\infty, \text{ as } j \to +\infty.
\end{equation}
Now, by our choice of $\delta$, for every $j \in \mathbb{N}$, we have that
\[
\bar\varphi + \varphi_{\delta,j} \in \mathcal{N}_\rho(\varphi),
\]
hence
\begin{equation}\label{contr2}
\|\det(H\bar\varphi + H\varphi_{\delta,j})\|_{\mathcal{H}^1(\Omega)} = \int_{\Omega}\det(H\bar\varphi + H\varphi_{\delta,j})\log(1 + \det(H\bar\varphi + H\varphi_{\delta,j})) \le m, \forall j \in \mathbb{N}.
\end{equation}
By the monotonicity of the determinant on the cone of non-negative definite symmetric matrices, we have
\[
\det(H\varphi_{\delta,j}) \le \det(H\bar\varphi + H\varphi_{\delta,j}), 
\]
and since the function $x \mapsto x\log(1 + x)$ is increasing for $x \ge 0$, then
\[
\|\det(H\varphi_{\delta,j})\|_{\mathcal{H}^1(\Omega)} \le \|\det(H\bar\varphi + H\varphi_{\delta,j})\|_{\mathcal{H}^1(\Omega)} \overset{\eqref{contr2}}{\le} m, \forall j \in \mathbb{N}.
\]
The last inequality is in contradiction with \eqref{expl2}.
\end{proof}

\bibliographystyle{plain}
\bibliography{DeterminantsDRT}

\end{document}